\documentclass[11pt]{article}
\usepackage[english]{babel}
\usepackage{amsmath,amsthm}
\usepackage{amsfonts}
\usepackage{bbm}
\usepackage{mathrsfs}
\usepackage{braket}
\usepackage{listings}
\usepackage[top=1in, bottom=1in, left=.5in, right=.5in]{geometry}
\usepackage{graphicx}
\usepackage[pdftex, colorlinks = true, citecolor = blue, linkcolor=black]{hyperref}
\usepackage{amssymb}
\usepackage{scrextend}
\usepackage{verbatim}
\date{}

\usepackage{color}

\definecolor{codegreen}{rgb}{0,0.6,0}
\definecolor{codegray}{rgb}{0.5,0.5,0.5}
\definecolor{codepurple}{rgb}{0.58,0,0.82}
\definecolor{backcolour}{rgb}{0.95,0.95,0.92}

\lstdefinestyle{mystyle}{
	backgroundcolor=\color{backcolour},   
	commentstyle=\color{codegreen},
	keywordstyle=\color{magenta},
	numberstyle=\tiny\color{codegray},
	stringstyle=\color{codepurple},
	basicstyle=\footnotesize,
	breakatwhitespace=false,         
	breaklines=true,                 
	captionpos=b,                    
	keepspaces=true,                 
	numbers=left,                    
	numbersep=5pt,                  
	showspaces=false,                
	showstringspaces=false,
	showtabs=false,                  
	tabsize=2
}

\newcommand{\Z}{\mathbb{Z}}
\newcommand{\N}{\mathbb{N}}

\newcommand{\SLZ}{SL_{2}(\Z)}

\newcommand*{\cofrac}[2]{%
	{%
		\rlap{$\dfrac{1}{\phantom{#1}}$}%
		\genfrac{}{}{0pt}{0}{}{#1+#2}%
	}%
} 

\newtheorem{lemma}{Lemma}

\newtheorem{conjecture}{Conjecture}

\begin{document}
	
	\title{An Analytic Heuristic for Multiplicity Computation for Zaremba's Conjecture}
	\author{Peter Cohen\\ Bowdoin College}
	\begin{titlepage}
		\maketitle
		\tableofcontents
		\thispagestyle{empty}
		
	\end{titlepage}
		
\setcounter{page}{1}

\section{Abstract}

Zaremba's Conjecture concerns the formation of continued fractions with partial quotients restricted to a given alphabet.  In order to answer the numerous questions that arrive from this conjecture, it is best to consider a semi-group, often denoted $\Gamma_{A}$, which arises naturally as a subset of $\SLZ$  when considering finite continued fractions.  To translate back from this semi-group into rational numbers, we select a projection mapping satisfying certain criteria to recover the numerator and denominator of the continued fractions in rational form.  The central question of our research is to determine the multiplicity of a given denominator.  To this end, we develop a heuristic method similar to the Hardy-Littlewood Circle Method.  We compare this theoretical model to the exact data, gleaned by simulation, and demonstrate that our formula appears to be asymptotically valid.  We then evaluate different aspects of the accuracy of our formula.  

\section{Introduction}

For any real number $\alpha \in [0, 1]$ we may write $\alpha$ as a continued fraction of the form
\begin{equation}
	\alpha = \cofrac{a_{1}}{
		\cofrac{a_{2}}{
			\cofrac{a_{3}}{
				\genfrac{}{}{0pt}{0}{}{\ddots.}}}} 
\end{equation} 

In this notation, each $a_{i}$ is called a partial quotient, and we will denote $\alpha$ by $[a_{1}, \ldots]$.  We restrict the possible values of $a_{i}$ to be in some alphabet $\mathcal{A} \subseteq \N$.  It is a well known fact that rational numbers have finite length continued fractions which are unique if restricted to an even number of partial quotients.

Zaremba's Conjecture \cite{ZAR72} states that there exists $A \in \N$ such that for all $q \in \N$ there exists $a \leq q$ that is co-prime to $q$ where $\frac{a}{q}$ has partial quotients bounded by $A$.\footnote{It should be noted that Zaremba's Conjecture has been proven for a density 1 subset of $\N$ by Kontorovich and Bourgain in \cite{KontBour14}.}  Thus, in the case of Zaremba's conjecture, the alphabet $\mathcal{A}$ is simply the set $\{1, \ldots, A\}$.  To study this conjecture, we rely upon the observation that, for $\frac{b}{d} = [a_{1}, \ldots, a_{n}]$
\begin{equation}
\begin{bmatrix}
* & b \\
* & d
\end{bmatrix} = \begin{bmatrix}
0 & 1 \\
1 & a_{1}
\end{bmatrix}	\begin{bmatrix}
0 & 1 \\
1 & a_{2}
\end{bmatrix} \cdots 	\begin{bmatrix}
0 & 1 \\
1 & a_{n}
\end{bmatrix} .
\end{equation}

	
Thus, it is natural to consider the set of \hyperref[Generating set S]{matrices} 
\begin{equation}
	 S = \Bigg\{ 	\begin{bmatrix}
			0 & 1\\
			1 & i
			\end{bmatrix}\Bigg\}_{i = 1}^{A} .
\end{equation}

This set can then be used to form all finite length continued fractions with partial quotients within $\mathcal{A} = \{1, \cdots, A\}$ by forming
\begin{equation}
	\Gamma_{A} = \Braket{S}^{+} \cap \SLZ,
\end{equation}
where $\Braket{S}^{+}$ denotes the semigroup generated by $S$.

Notice that the restriction imposed by intersecting $\Braket{S}^{+}$ with $\SLZ$ causes all elements of $\Gamma_{A}$ to be of an even number of partial quotients; however, as noted before, all rationals may be expressed in this form, and therefore this restriction imposes no restriction on the set of fractions that may be studied using this semi-group $\Gamma_{A}$.
 
We will say that a projection $f: \Gamma_{A} \rightarrow \N$ satisfies the \textbf{local condition} if for any given $m \in \N$ and any $\gamma_{1}, \gamma_{2} \in \Gamma_{A}$
\begin{equation}\label{Local Condition}
	\gamma_{1} \equiv \gamma_{2} \mod{m} \implies f(\gamma_{1}) \equiv f(\gamma_{2}) \mod{m}.
\end{equation}

For our purposes, the projection $f$ selects the bottom-right element of a matrix in $\Gamma_{A}$.  In the language of continued fractions, this projection amounts to selecting the denominator of a rational number represented by a matrix in $\Gamma_{A}$.  The projection $f$ clearly satisfies the local condition.

Through this paper, we seek to address the multiplicity of a particular value $n \in f(\Gamma_{A}) $ where $f(\Gamma_{A})$ is the image of $\Gamma_{A}$ under $f$.  We define multiplicity precisely as
\begin{equation}\label{first definition of multiplicity}
	mult(n) = |f^{-1}(n)|.
\end{equation}

Observe that, for any $n \in \N$, we are guaranteed to have a finite value of $mult(n)$.  This observation comes from the fact that for any given $w \in \Gamma_{A}$ of the form
\[ 
	w = \begin{bmatrix}
		a & b \\
		c & d
		\end{bmatrix}
\]
the largest entry of $w$ is always $d$.  Additionally, acting on $w$ by any element of $\Gamma_{A}$ by right matrix multiplication will form a new matrix $w' \in \Gamma_{A}$ such that the bottom-right element of $w'$ is strictly larger than the bottom-right element of $w$.  A useful consequence of this is that, in the Cayley graph of $\Gamma_{A}$ given with respect to even length products in $S$, only elements on the sphere of radius $n$ in the supremum norm contribute to the multiplicity of a value $n$.  This is a crucial fact to be used later.

From this framework, we build the conjecture that the multiplicity of a given value of $n$ in the image of $\Gamma_{A}$ under the projection mapping that selects the bottom-right entry of $\gamma \in \Gamma_{A}$ is 
\begin{equation}
mult(n) \sim \frac{2 \delta \Big| B_{n}(\Gamma_{A}) \Big|}{n}\prod_{p |n} \Big(\frac{p - 1}{p} \Big) \zeta(2),
\end{equation}
where $\Big| B_{n}(\Gamma_{A}) \Big|$ is the number of elements in the ball of radius $n$ in $\Gamma_{A}$ under the Archimedean metric.

To support this conjecture, we build an approximation to $mult(n)$ via a method similar to the Hardy-Littlewood Circle Method, which is detailed further below.  Rather than using the full extent of Circle Method, we take an approximation to the singular integral and show computationally that this more tractable version is still effective at evaluating the multiplicities of sufficiently large $n$.  We develop a singular series as a critical component of the approximation to $mult(n)$.

We present evidence that our heuristic is reasonable based upon the average value of the singular series.  Additionally, using large-scale computer simulation, we demonstrate that our approximation appears to be asymptotically valid.  In other words, from our computer simulation, as $n$ grows below the bound $9 \times 10^{5}$, our approximation to $mult(n)$ consistently improves.

This paper presents the construction of the heuristic and then gives computational evidence suggesting the validity of the asymptotic approximation.  In \S \ref{sec: Formation of the Singular Series} we construct the singular series via successive approximations to $mult(n)$.  In \S \ref{sec: Evaluation of the Singular Series} we present a proof of \cite{Huang} that is necessary to simplify the singular series in order to make the paper self-contained.  In \S \ref{sec: Evaluation of the Singular Series} we then use this result of \cite{Huang} to explicitly compute the value of the singular series and, consequently, the value of the multiplicity approximation; additionally, we show that the singular series averages to unity in the limit, which further supports our multiplicity conjecture.  In \S \ref{sec: Computational Methodology} we outline the algorithm used to precisely compute multiplicities for a bounded subset of $\Gamma_{A}$, and in \S \ref{sec: Computational Results} the exact computational results are compared to the approximate results given from our heuristic in order to provide additional justification for the asymptotic validity of our conjecture.  

As this paper builds upon the concept of the Hardy-Littlewood Circle Method, we direct the unfamiliar reader to \cite{MTB06} for a general treatment and suggest \cite{MN96} as an illustrative example of the Circle Method's application to Waring's Problem.

\section{Formation of the Singular Series}\label{sec: Formation of the Singular Series}
We desire to compute the multiplicity of a particular value $n$, this is exactly given by
\begin{equation}
mult(n) = \sum_{\gamma \in \Gamma_{A}} \mathbbm{1}_{ \{f(\gamma) - n = 0\} }.
\end{equation}  

We define the norm of an element $g \in \Gamma_{A}$ to be the maximum of the entries of $g$.  For the sake of notation, let $S_{n}(\Gamma_{A})$ be the sphere of norm $n$ in $\Gamma_{A}$.  As noted in the introduction, we need only consider semi-group elements in $S_{n}(\Gamma_{A})$ as only elements of this sphere can contribute to the multiplicity of $n$.  Thus,
\begin{equation}\label{eq: Exact multiplicity equation}
	mult(n) = \sum_{\gamma \in S_{n}(\Gamma_{A})} \mathbbm{1}_{ \{f(\gamma) - n = 0\} }.
\end{equation}

For the sake of notation, let $m(\gamma) = f(\gamma) - n$.

Recalling that 
\begin{equation}
	\int_{0}^{1}e(mx)e(-nx) \, dx = \begin{cases}
										0, & m \neq n\\
										1, & m = n
									\end{cases},
\end{equation}
for $e(mx) = e^{2 m \pi \imath x}$, \eqref{eq: Exact multiplicity equation} may be rewritten as 
\begin{equation}
\sum_{\gamma \in S_{n}(\Gamma_{A})} \int_{0}^{1} e\Big(xm(\gamma) \Big) \, dx.
\end{equation}

We may commute the sum and integral because $S_{n}(\Gamma_{A})$ contains only finitely many elements; this yields
\begin{equation} \label{eq: Exact multiplicicty equation integral}
\int_{0}^{1} \sum_{\gamma \in S_{n}(\Gamma_{A})} e\Big(xm(\gamma) \Big) \, dx.
\end{equation}

Instead of computing this integral, we attempt a heuristic approximation method that is similar in concept to that of the Circle Method.  Essentially, we approximate the integral at hand by only computing the singular series of the Circle Method; we then use simulation to show close asymptotic agreement to the precise multiplicity values.  To  encapsulate this concept, we fix some $Q \in \N$ as an upper bound on the denominators of rational numbers and then substitute for the integral in \eqref{eq: Exact multiplicicty equation integral} yielding

\begin{equation}\label{eq: First approximation of multiplicity}
\sum_{q < Q} \sum_{(a,q) = 1} \sum_{\gamma \in S_{n}(\Gamma_{A})} e \Big(\frac{a m(\gamma)}{q} \Big).
\end{equation}

Because $f$ satisfies the local condition, we may rewrite \eqref{eq: First approximation of multiplicity} in terms of the residue classes of $\Gamma_{A}$ modulo $q$.  Hence, \eqref{eq: First approximation of multiplicity} is presented equivalently as
\begin{equation}\label{eq: Residue form rewriting of first approximation}
\sum_{q < Q} \sum_{(a,q) = 1} \sum_{\gamma_{0} \in \Gamma_{A} (\text{mod }q)} e \Big(\frac{a m(\gamma_{0})}{q} \Big) \sum_{\gamma \in S_{n}(\Gamma_{A}) \atop \gamma \equiv \gamma_{0} \mod{q}}1.
\end{equation}

To simplify \eqref{eq: Residue form rewriting of first approximation}, we assume that $\Gamma_{A}$ is equidistributed over the residue classes; this assumption is examined in further depth in \S \ref{sec: Equidistribution}.That is to say that the number of elements whose projection under $f$ relates to one particular residue class is roughly the same as the number of elements whose projection under $f$ relates to another residue class.

Assuming equidistribution, \eqref{eq: Residue form rewriting of first approximation} becomes approximately
\begin{equation}\label{eq: Base equation for singular series}
\sum_{q < Q} \sum_{(a,q) = 1} \sum_{\gamma_{0} \in \Gamma_{A} (\text{mod }q)} e \Big(\frac{a m(\gamma_{0})}{q} \Big) \frac{\Big| S_{n}(\Gamma_{A}) \Big|}{\Big| \Gamma_{A} (\text{mod }q) \Big|}.
\end{equation}

Noticing that $\Big| S_{n}(\Gamma_{A}) \Big|$ has no dependency upon the values of $q$ or $a$, it may be factored out of the sum entirely.  Additionally, because $\Big| \Gamma_{A} (\text{mod }q) \Big|$ is dependent upon neither $\gamma_{0}$ nor $a$ it may be factored out of the second and third summations.  Thus, \eqref{eq: Base equation for singular series} becomes

\begin{equation}\label{eq: Reversed order of summation base equation}
\Big| S_{n}(\Gamma_{A}) \Big| \sum_{q < Q} \Bigg( \frac{1}{\Big| \Gamma_{A} (\text{mod }q) \Big|} \Bigg) \sum_{\gamma_{0} \in\Gamma_{A}(\text{mod }q)} \sum_{(a,q) = 1}  e \Big(\frac{a m(\gamma_{0})}{q} \Big).
\end{equation}

Recalling the Ramanujan sum 
\begin{equation}\
	c_{q}(n) = \sum_{(a, q) = 1}e \Big(\frac{a n}{q} \Big),
\end{equation}

we see that \eqref{eq: Reversed order of summation base equation} may be written as
\begin{equation}\label{eq: Ramanujan sum form of base equation}
\Big|  S_{n}(\Gamma_{A}) \Big| \sum_{q < Q} \Bigg( \frac{1}{\Big| \Gamma_{A} (\text{mod }q) \Big|} \Bigg) \sum_{\gamma_{0} \in\Gamma_{A}(\text{mod }q)} c_{q}(m(\gamma_{0})).
\end{equation}

It follows from the work of Hensley in \cite{Hensley89} that $|B_{n}(\Gamma_{A})| \asymp cn^{2 \delta}$  where $B_{n}(\Gamma_{A})$ is the ball of radius $n$ in $\Gamma_{A}$ under the Archimedean metric and $c$ and $\delta$ are constants depending on $\Gamma_{A}$.\footnote{See \cite{Hensley92} for further details concerning the value of $\delta$.}  Heuristically, $|S_{n}(\Gamma_{A})| \approx 2 \delta cn^{2 \delta - 1}$, and
\begin{equation}
	|S_{n}(\Gamma_{A})| \approx \frac{2 \delta |B_{n}(\Gamma_{A})|}{n}.
\end{equation}

From \eqref{eq: Ramanujan sum form of base equation} we construct the \emph{singular series}
\begin{equation}\label{eq: G(n)}
\mathfrak{G}(n) = \sum_{q= 1}^{\infty} \Bigg( \frac{1}{\Big| \Gamma_{A} (\text{mod }q) \Big|} \Bigg) \sum_{\gamma_{0} \in\Gamma_{A}(\text{mod }q)} c_{q}(m(\gamma_{0})).
\end{equation}

To simplify this equation, let 
\begin{equation} \label{eq: bar{C}_{q}}
	\bar{C}_{q}(\Gamma_{A}, n) = \Bigg( \frac{1}{\Big| \Gamma_{A} (\text{mod }q) \Big|} \Bigg) \sum_{\gamma_{0} \in \Gamma_{A}(\text{mod }q)} c_{q}(f(\gamma_{0}) - n).
\end{equation}

Note that the group $\Braket{\Gamma_{A}}$ is all of $\SLZ$, and the reduction of $\Braket{\Gamma_{A}}$ modulo $q$ is $\text{SL}_{2}(\Z / q\Z)$.  This is an elementary instance of strong approximation, which is treated in detail in \cite{Huang}.

Then, by multiplicativity and unique factorization \eqref{eq: G(n)} has the Euler product
\begin{equation} \label{Singular Series Preliminary Euler Product}
	\mathfrak{G}(n) = \prod_{p}\Bigg( 1 + \sum_{k = 1}^{\infty}\bar{C}_{p^{k}}\Big( \Gamma_{A}, n \Big) \Bigg).
\end{equation}

\section{Evaluation of the Singular Series}\label{sec: Evaluation of the Singular Series}
In order to evaluate \eqref{eq: G(n)}, one must first compute $\sum_{k = 1}^{\infty}\bar{C}_{p^{k}}\Big( \Gamma_{A}, n \Big)$.  Thus, $\bar{C}_{q}(\Gamma_{A}, n)$ need only be computed for prime power values of $q$.  Thus, consider $q = p^{t}$ where $p$ is a prime and $t \in \N$.

Following from \cite{Huang} we have

\begin{equation}\label{eq: evaluation of C}
\bar{C}_{p^{t}}(\Gamma, n) =  	\begin{cases}
\frac{-1}{p + 1}, & \text{if } t = 1, p|n \\
\frac{1}{p^{2} - 1}, & \text{if } t = 1, p \nmid n \\
0, & \text{if } t \geq 2 \\
\end{cases}.
\end{equation}

For the reader's convenience we reprove \eqref{eq: evaluation of C}
\begin{proof}
	\textbf{Case $t = 1, p|n$:}
	
	Recall that M\"{o}bius inversion of the Ramanujan sum $c_{q}(f(\gamma_{0}) - n)$ gives that 
	\begin{equation}
		c_{q}(f(\gamma_{0}) - n) = \sum_{s|(p, d-n)} s \mu\Big(\frac{p}{s}\Big).
	\end{equation}
	
	This allows us to rewrite the sum in \eqref{eq: bar{C}_{q}} as 
	\begin{align}
	\sum_{w \in SL_{2}(\Z/p\Z)}\sum_{s|(p, d-n)} s \mu\Big(\frac{p}{s}\Big) &= \sum_{w \in SL_{2}(\Z/p\Z) \atop (d - n, p) = 1} \mu(p) + \sum_{w \in SL_{2}(\Z/p\Z) \atop d \equiv n (p)} p\mu(1) + 1\mu(p)\nonumber\\
	&= -\sum_{w \in SL_{2}(\Z/p\Z) \atop (d, p) = 1} 1 + (p - 1)\sum_{w \in SL_{2}(\Z/p\Z) \atop d \equiv 0 (p)}1.
	\end{align}
	
	It is easily computed that the number of elements in $SL_{2}(\Z/p\Z)$ for which $(d, p) = 1$ is $p^{3} - p^{2}$ and that the number of elements in $SL_{2}(\Z/p\Z)$ for which $d \equiv 0 (p)$ is $p^{2} - p$.
	
	Thus,
	\begin{equation}
	 -\sum_{w \in SL_{2}(\Z/p\Z) \atop (d, p) = 1} 1 + (p - 1)\sum_{w \in SL_{2}(\Z/p\Z) \atop d \equiv 0 (p)}1 = -p^{2}(p - 1) + p(p - 1)^{2}.
	 \end{equation}
	
	And so, 
	\begin{equation} 
	\bar{C}_{p}(\Gamma, n) =   \Bigg( \frac{1}{\Big|SL_{2}(\Z / p\Z) \Big|} \Bigg) \Big(-p^{2}(p - 1) + p(p - 1)^{2}\Big) = \frac{-1}{p + 1}.
	\end{equation}
	
	\textbf{Case $t = 1, p\nmid n$:}
	Again, 
	\begin{align}
	\sum_{w \in SL_{2}(\Z/p\Z)}\sum_{s|(p, d-n)} s \mu\Big(\frac{p}{s}\Big) &= \sum_{w \in SL_{2}(\Z/p\Z) \atop (d - n, p) = 1} \mu(p) + \sum_{w \in SL_{2}(\Z/p\Z) \atop d \equiv n (p)} p\mu(1) + 1\mu(p)\nonumber\\
	&= -\sum_{w \in SL_{2}(\Z/p\Z) \atop (d - n, p) = 1} 1 + (p - 1)\sum_{w \in SL_{2}(\Z/p\Z) \atop d \equiv n (p)}1.
	\end{align}
	Note that the number of elements in $SL_{2}(\Z/p\Z)$ for which $(d - n, p) = 1$ is $p^{3} - p^{2} - p$ and that the number of elements in $SL_{2}(\Z/p\Z)$ for which $d \equiv n (p)$ is $p^{2}$.
	
	Thus,
	\begin{equation} 
	-\sum_{w \in SL_{2}(\Z/p\Z) \atop (d - n, p) = 1} 1 + (p - 1)\sum_{w \in SL_{2}(\Z/p\Z) \atop d \equiv n (p)}1 = p .
	\end{equation}
	
	Hence, 
	\begin{equation}
	 \bar{C}_{p}(\Gamma, n) =   \Bigg( \frac{1}{\Big|SL_{2}(\Z / p\Z) \Big|} \Bigg) \Big(p\Big) = \frac{1}{p^{2} - 1}.
	 \end{equation}
	
	\textbf{Case $t > 1$:}
	In this case, the only values of $s$ for which $\mu(\frac{p^{t}}{s})$ is non-zero are $p^{t}$ and $p^{t - 1}$.  Thus, 
	
	\begin{align}
	\sum_{w \in SL_{2}(\Z/p^{t}\Z)}\sum_{s|(p^{t}, d-n)} s \mu\Big(\frac{p^{t}}{s}\Big) &= \sum_{w \in SL_{2}(\Z/p^{t}\Z) \atop d \equiv n (p^{t})} p^{t}\mu(1) + p^{t - 1}\mu(p) + \sum_{w \in SL_{2}(\Z/p^{t}\Z) \atop d \equiv n (p^{t - 1}); d \not\equiv n (p^{t})} p^{t - 1}\mu(p)\nonumber\\
	&= p^{t - 1}\left((p - 1)\sum_{w \in SL_{2}(\Z/p^{t}\Z)}\mathbbm{1}_{\{d \equiv n (p^{t}) \}}- \sum_{w \in SL_{2}(\Z/p^{t}\Z)}\mathbbm{1}_{ 		\begin{Bmatrix}
		d \equiv n (p^{t - 1}) \\
		d \not\equiv n (p^{t})
		\end{Bmatrix} } \right).
	\end{align}
	
	All that remains to be shown is that 
	\begin{equation} 
	(p - 1)\sum_{w \in SL_{2}(\Z/p^{t}\Z)}\mathbbm{1}_{\{d \equiv n (p^{t}) \}} = \sum_{w \in SL_{2}(\Z/p^{t}\Z)}\mathbbm{1}_{ 	\begin{Bmatrix}
		d \equiv n (p^{t - 1}) \\
		d \not\equiv n (p^{t})
		\end{Bmatrix} }.
	\end{equation}
	
	This reduces to a simpler problem.  Namely, it must only be shown that, for $\gamma \equiv \begin{bmatrix}
	* & * \\
	* & n
	\end{bmatrix} \mod{p^{t - 1}}$
	\begin{equation}\label{Shinnyih's 5.32}
	(p - 1)\sum_{w \in SL_{2}(\Z/p^{t}\Z) \atop w \equiv \gamma (p^{t - 1})}\mathbbm{1}_{\{d \equiv n (p^{t}) \}} = \sum_{w \in SL_{2}(\Z/p^{t}\Z) \atop w \equiv \gamma (p^{t - 1})}\mathbbm{1}_{\{d \not\equiv n (p^{t})\}}.
	\end{equation}
	
	Given a $w \in SL_{2}(\Z/p^{t}\Z)$ so that $w \equiv \begin{bmatrix}
	a_{1} & b_{1} \\
	c_{1} & n 
	\end{bmatrix} (p^{t - 1})$ the matrix $w$ may be rewritten as 
	\begin{equation} 
	w =  \begin{bmatrix}
	a_{1} + p^{t - 1}k_{1} & b_{1} + p^{t - 1}k_{2} \\
	c_{1} + p^{t - 1}k_{1} & n + p^{t - 1}k_{1} 
	\end{bmatrix},
	\end{equation}
	
	for which $0 \leq k_{i} < p$ and $p|\frac{det(w)}{p^{t - 1}}$.  For any choice of $k_{4}$, the triplet $(k_{1}, k_{2}, k_{3})$ may have $p^{2}$ values.  Thus, both the right-hand side and left-hand side of \eqref{Shinnyih's 5.32} agree as they both hold the value $(p - 1)p^{2}$.  This concludes the proof.
\end{proof}

Armed with this information, we may now put $\mathfrak{G}(n)$ into a more practical form.  Simply substituting for appropriate values of $\bar{C}_{p^{t}}(\Gamma_{A}, n)$ in \eqref{Singular Series Preliminary Euler Product} gives
\begin{align}
\mathfrak{G}(n) = \prod_{p} \Big(1 + \sum_{q = 1}^{\infty} \bar{C}_{p^{q}}(\Gamma, n) \Big) &=  \prod_{p |n} \Big(1 + \frac{-1}{p + 1} \Big) \prod_{p\, \nmid \,n} \Big(1 + \frac{1}{p^{2} - 1} \Big)\nonumber \\
&= \prod_{p |n} \Big(1 + \frac{-1}{p + 1} \Big) \Big(1 + \frac{1}{p^{2} - 1} \Big)^{-1} \prod_{p}\Big(1 + \frac{1}{p^{2} - 1} \Big)\nonumber\\
&= \prod_{p |n} \Big(\frac{p - 1}{p} \Big) \zeta(2).
\end{align}

Based upon our construction we form the following conjecture
\begin{conjecture}\label{conj: mult conjecture}
	The multiplicity of a given value of $n$ in the image of $\Gamma_{A}$ under the projection mapping that selects the bottom-right entry of $\gamma \in \Gamma_{A}$ is 
	\begin{equation}
	mult(n) \sim \frac{2 \delta \Big| B_{n}(\Gamma_{A}) \Big|}{n}\prod_{p |n} \Big(\frac{p - 1}{p} \Big) \zeta(2).
	\end{equation}
\end{conjecture} 

If our conjecture is indeed valid, then it must be the case that 
\begin{equation}\label{eq: Sum of multiplicities conjecture}
\Big| B_{N}(\Gamma_{A}) \Big| \sim \sum_{n = 1}^{N} \frac{2 \delta \Big| B_{n}(\Gamma_{A}) \Big|}{n}\prod_{p |n} \Big(\frac{p - 1}{p} \Big) \zeta(2).
\end{equation}

as $|B_{N}(\Gamma_{A})| = \sum_{n = 1}^{N} mult(n)$.  To show \eqref{eq: Sum of multiplicities conjecture} we will show that $\mathfrak{G}(n)$ averages to 1 in the limit as $N \rightarrow \infty$ and then we will use summation by parts to conclude the justification.

\begin{lemma}\label{lem: G(n) averages to 1}
	$\mathfrak{G}(n)$ averages to 1 in the limit as $n \rightarrow \infty$.  Formally,
	\begin{equation}
	\lim\limits_{N \rightarrow \infty} \frac{1}{N} \sum_{n = 1}^{N} \mathfrak{G}(n) = 1.
	\end{equation}
\end{lemma} 
\begin{proof}
	\begin{align}
	\frac{1}{N} \sum_{n = 1}^{N} \mathfrak{G}(n) &= \frac{1}{N} \sum_{n = 1}^{N} \prod_{p|n}\big(\frac{p - 1}{p} \big)\big(\frac{\pi^{2}}{6} \big)\nonumber\\
	&= \frac{1}{N} \sum_{n = 1}^{N} \sum_{m|n}\big(\frac{\mu(m)}{m} \big)\big(\frac{\pi^{2}}{6} \big)\nonumber\\
	&= \frac{1}{N} \sum_{m \leq N} \sum_{n \equiv 0 (m) \atop n \leq N} \big(\frac{\mu(m)}{m} \big)\big(\frac{\pi^{2}}{6} \big)\nonumber\\
	&= \big(\frac{\pi^{2}}{6} \big) \frac{1}{N} \sum_{m \leq N} \big(\frac{\mu(m)}{m} \big) \left[\frac{N}{m}\right]\nonumber\\
	&\approx \big(\frac{\pi^{2}}{6} \big) \sum_{m \leq N} \big(\frac{\mu(m)}{m^{2}} \big).
	\end{align}
	
	Because $\sum_{m = 1}^{\infty} \big(\frac{\mu(m)}{m^{2}} \big) = \frac{1}{\zeta(2)}$, 
	\begin{equation}
	\lim\limits_{N \rightarrow \infty} \big(\frac{\pi^{2}}{6} \big) \sum_{m \leq N} \big(\frac{\mu(m)}{m^{2}} \big) = \big(\frac{\pi^{2}}{6} \big) \frac{1}{\zeta(2)} = 1.
	\end{equation}
\end{proof}

Using Lemma~\eqref{lem: G(n) averages to 1} we see that $\sum_{n = 1}^{N}\mathfrak{G}(n) \sim N$.  Summing $\sum_{n = 1}^{N} \frac{2 \delta | B_{n}(\Gamma_{A}) |}{n}\mathfrak{G}(n)$ by parts yields,

\begin{align}
\sum_{n = 1}^{N} \frac{2 \delta \Big| B_{n}(\Gamma_{A}) \Big|}{n}\mathfrak{G}(n) &= \sum_{n = 1}^{N} \frac{2 \delta cn^{2\delta}}{n}\mathfrak{G}(n)\nonumber \\
&= 2\delta c \Big((N + 1)^{2\delta - 1}(N + 1) - 1 - \sum_{n = 1}^{N}(n + 1)\Big((n + 1)^{2\delta - 1} - n^{2\delta - 1} \Big) \Big)\nonumber\\
&= 2\delta c\Big((N + 1)^{2\delta} - 1 - \sum_{n = 1}^{N}\Big((n + 1)^{2\delta} - n^{2\delta} - n^{2\delta - 1} \Big) \Big)\nonumber\\
&= cN^{2\delta}.
\end{align}

It then follows that \eqref{eq: Sum of multiplicities conjecture} holds.

\section{Computational Methodology}\label{sec: Computational Methodology}
In order to test the validity of Conjecture~\eqref{conj: mult conjecture} we developed an algorithm to efficiently compute multiplicities for a large range of target values.  The algorithm functions by forming the set \hypertarget{Generating set S}{$S$} using $2 \times 2$ arrays.  In order to compute only even length products of elements in $S$, we create $S_{2}$ by forming $S \times S$ and mapping each $(s_{1}, s_{2}) \in S \times S$ to $s_{1}s_{2}$.  We then consider the semi-group generated by $S_{2}$; clearly this is equivalent to our old formulation of $\Gamma_{A}$.  It ought to be noted that for computation purposes, $A$ was set to be 5.  This choice of $A$ is not arbitrary; in fact it was proposed by Zaremba.  Specifically, for $A \in \{1, 2, 3, 4\}$ there are examples showing that Zaremba's Conjecture fails.  For greater detail of this point we direct the reader to \cite{Kont13}.

It may come to mind that the Cayley graph of $\Gamma_{A}$ with respect to $S_{2}$ resembles an $A^{2}$-ary tree; this is not strictly true.  The reason that this is not true is because the identity matrix $I_{2 \times 2}$ is not an element of $\Gamma_{A}$; however, setting $I_{2 \times 2}$ as the root of an $A^{2}$-ary tree with all non-root nodes elements of $\Gamma_{A}$ allows us to perform a highly efficient recursive algorithm on this modified Cayley graph in order to both build and tally multiplicities simultaneously.  



\section{Computational Results}\label{sec: Computational Results}
To test the asymptotic behavior of our heuristic for multiplicities, we computed multiplicities for target values ranging from $1$ to $9 \times 10^{5}$.  We then directly evaluated $\frac{2 \delta | B_{n}(\Gamma_{A}) |}{n}\mathfrak{G}(n)$ for each of these values as well.  Under our conjecture, the ratio of the true multiplicity of a given target to the heuristically calculated value for the same target ought to limit to 1.  Upon inspection, it appears that this is the case.
\begin{figure}[h!]
	\centering
	\includegraphics[width = 4.5in]{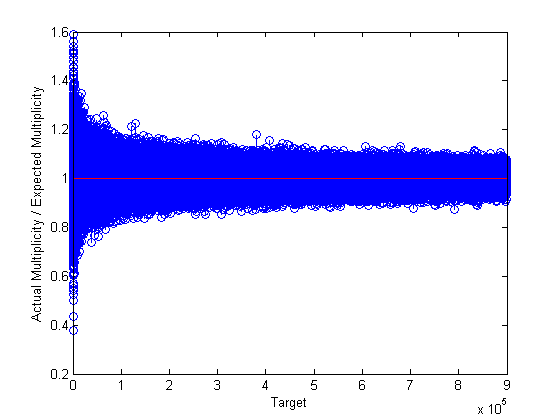}
	
	\caption{Multiplicity vs. heuristic ratios for target values from 1 to $9 \times 10^{5}$}
	\label{fig:MultRatios}
\end{figure}

\newpage
\section{Conclusion}
To conclude, we note that a great deal of work remains to be done.  Namely, in order to conform more fully to the Circle Method, an appropriate treatment of the major arc analysis is necessary.  For our purposes, we approximated the integral in \eqref{eq: Exact multiplicicty equation integral} using point masses; however, analysis of the major arcs may lead to further accuracy of the resulting heuristic.

\section{Acknowledgments}

Many thanks to Alex Kontorovich for his guidance.  Additionally, we acknowledge support from Kontorovich's NSF grants  DMS-1209373, DMS-1064214, DMS-1001252, and Kontorovich's NSF CAREER grant DMS-1254788

\appendix
\section{Equidistribution of $\Gamma_{A}$ Over Residue Classes}\label{sec: Equidistribution}
In \S \ref{sec: Formation of the Singular Series}, we rely upon the equidistribution of the $\Gamma_{A}$ over residue classes modulo $q$ in order to state that 
\begin{equation}\label{eq: equidistribution equation}
	\sum_{q < Q} \sum_{(a,q) = 1} \sum_{\gamma \in S_{n}(\Gamma_{A})} e \Big(\frac{a m(\gamma)}{q} \Big) = \sum_{q < Q} \sum_{(a,q) = 1} \sum_{\gamma_{0} \in \Gamma_{A} (\text{mod }q)} e \Big(\frac{a m(\gamma_{0})}{q} \Big) \sum_{\gamma \in S_{n}(\Gamma_{A}) \atop \gamma \equiv \gamma_{0} \mod{q}}1.
\end{equation}

It is not immediately obvious that $f(\Gamma_{A})$ is equidistributed over residue classes.  In fact, for small values of $q$, this is not the case.  However, simulation of the distribution of $\Gamma_{A}$ over residue classes modulo $q$ for increasing values of $q$ suggests that this is likely asymptotically true.

In order to simulate the distribution of $\Gamma_{A}$ over residue classes, the same computational methodology was employed as in \S \ref{sec: Computational Methodology} to compute the values of $T \mod{q}$ for $T$ a finite bounded subset of $\Gamma_{A}$.\footnote{For the sake of consistency, we continue to simulate using $A = 5$.}  However, in this implementation, we store which residue class each matrix falls into for moduli ranging from $1$ to $30$.  Then, normalizing by the total number of matrices in $T$, we compute the absolute error between the observed residue count modulo $m$ and the expected residue count under the assumption of equidistribution over residue classes modulo $m$.  Formally, we compute the absolute error as
\begin{equation}
	\text{Absolute Error}(r,m) := \left| \#\left\{\text{matrices in $T$ with residue }r \text{ mod } m \right\} - \frac{|T|}{m}\right|.
\end{equation}

Then, we define
\begin{equation}▲
	\text{Largest Error}(m) := \max\limits_{0 \leq r < m} \left\{ \text{Absolute Error}(r,m) \right\}.
\end{equation}

\newpage
Simulation of $\text{Largest Error}(m)$ for $1 \leq m \leq q$ shows that as $q$ grows, the greatest absolute error between the observed distribution and the expected equidistribution shrinks.  This is encapsulated in the following figure.

\begin{figure}[h!]\label{fig: Equidistribution Image}
	\centering
	\includegraphics[width = 4.5in]{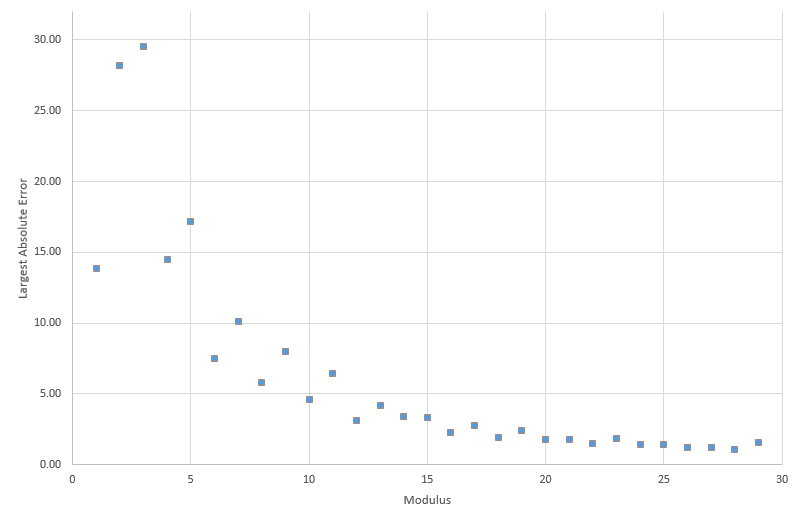}
	
	\caption{Largest absolute error plotted against modulus.}
\end{figure}

As the largest absolute error away from equidistribution shrinks rapidly, the assumption of equidistribution is supported in the limit, and thus  \eqref{eq: equidistribution equation} is supported.

\bibliographystyle{alpha}
\bibliography{FinalWriteupBibliography}

\section{Contact Information}
Peter Cohen\\
\href{mailto:plcohen@mit.edu}{plcohen@mit.edu}\\
609-571-5102


\end{document}